\documentclass{article}%
\usepackage{amsmath}
\usepackage{amsfonts}
\usepackage{amssymb}
\usepackage{graphicx}%
\providecommand{\U}[1]{\protect\rule{.1in}{.1in}}
\newtheorem{theorem}{Theorem}
\newtheorem{corollary}{Corollary}
\newtheorem{lemma}{Lemma}
\newtheorem{proposition}{Proposition}
\newenvironment{proof}[1][Proof]{\noindent\textbf{#1.} }{\ \rule{0.5em}{0.5em}}
\begin{document}

\title{Remarks on profinite groups having few open subgroups}
\author{Dan Segal}
\maketitle

A profinite group is \emph{small} if for each $n\in\mathbb{N}$ it has only
finitely many open subgroups of index at most $n$.

Every finitely generated profinite group is small. Small groups also arise in
number theory: if $S$ is a finite set of primes and $K$ is the maximal
algebraic extension of $\mathbb{Q}$ unramified outside $S$ then $\mathrm{Gal}%
(K/\mathbb{Q})$ is a small profinite group (\cite{K}, Theorem 1.48) (whether
all such Galois groups are in fact finitely generated is apparently an open problem).

If $G$ is a finitely generated profinite group, then (a) every subgroup of
finite index is open and (b) every power subgroup $G^{m}$ is open (\cite{NS1},
\cite{NS2}; for better proofs see also \cite{NS3}); here $G^{m}=\left\langle
g^{m}\mid g\in G\right\rangle $ denotes the subgroup generated
\emph{algebraically} (not topologically) by all $m$th powers in $G$.

If (a) holds one says that $G$ is \emph{strongly complete}. If (b) holds I
will say that $G$ is \emph{power-open}. It is clear that (b) implies (a).

We shall see below that every strongly complete group is small. A small group
need be neither strongly complete nor power-open; the purpose of this note
(which is largely a recap of known results) is to explore some connections
between these various concepts, in particular, to what extent they can be
`algebraically defined'. Writing
\[
\mathcal{F}(G)=\{G/N\mid N\vartriangleleft_{o}G\}
\]
to denote the family of all continuous finite quotients of a profinite group
$P$, I will say that a property of $P$ is \emph{algebraically defined} if it
can be stated in terms of some purely group-theoretic property of the groups
in $\mathcal{F}(P)$ -- this may seem slightly vague but will be clear in the
cases discussed below.

Significant contibutions are due to Nikolay Nikolov and John Wilson; thanks to
both for allowing me to quote some unpublished results.

I will use the following notation. For subset $X$ of a group,%
\[
X^{\ast n}=\{x_{1}x_{2}\ldots x_{n}\mid x_{1},x_{2},\ldots,x_{n}\in X\}.
\]
For a group word $w$ on $k$ variables,%
\[
G_{w}=\{w(\mathbf{g})^{\pm1}\mid\mathbf{g}\in G^{(k)}\},~~w(G)=\left\langle
G_{w}\right\rangle ;
\]
and for $m\in\mathbb{N}$, $G_{\{m\}}=\{g^{m}\mid g\in G\}$, $G^{m}%
=\left\langle G_{\{m\}}\right\rangle $.

$\overline{X}$ denotes the closure of a subset $X$ in a profinite group $G$.
We write $N\vartriangleleft_{o}G$ to mean: $N$ is an open normal subgroup of
$G$.

The word $w$ has \emph{width} $f$ in $G$ if $w(G)=G_{w}^{\ast f}$, and
infinite width if this holds for no finite $f$. We recall that in a profinite
group $G$, the subgroup $w(G)$ is closed if and only if $w$ has finite width
in $G$; this holds if and only if $w$ has bounded width in $\mathcal{F}(G)$
(see \cite{S}, Section 4.1). If $w(G)$ has countable index in $G$ then $w(G)$
is open, hence has finite index (\cite{SW}, Lemma 2).

\section{Examples}

In the proof of \cite{N2}, Theorem 4, Nikolov introduces a general method for
constructing groups with large verbal width. The basic idea is summed up in
the next lemma.

For a group $B$ let $\mathcal{S}_{n}(B)$ denote the set of all $n$-generator
subgroups of $B$.

\begin{lemma}
\label{wordexample}Let $w$ be a word in $k$ variables and let $G=M\rtimes B$
be a semi-direct product with $w(M)=1$. Suppose that for each $H\in
\mathcal{S}_{km}(B)$ of $B$ we have $M=A_{H}\times D_{H}$ with $[A_{H},H]=1$
and $[D_{H},H]\leq D_{H}$. Then for any $\mathbf{g}_{1},\ldots,\mathbf{g}%
_{m}\in G^{(k)}$ there exists $H\in\mathcal{S}_{km}(B)$ such that
\begin{equation}
\prod\nolimits_{1}^{m}w(\mathbf{g}_{i})^{\pm1}\in D_{H}\cdot H. \label{(1)}%
\end{equation}

\end{lemma}

This is clear: take $H=\left\langle b_{ij}\mid i=1,\ldots,m,~j=1,\ldots
,k\right\rangle $ where $g_{ij}\in Mb_{ij}$, $b_{ij}\in B$, and observe that%
\[
w(\mathbf{g}_{i})\in w(A_{H}\times(D_{H}\cdot H)).
\]
Now suppose that%
\[
w(G)\supseteq M\neq\bigcup\limits_{H\in\mathcal{S}_{km}(B)}D_{H}.
\]
Then some element of $M$ is not of the form (\ref{(1)}), and it follows that
$w$ does not have width $m$ in $G$.

\begin{proposition}
Let $\pi$ be a non-empty set of primes with infinite complement. There exists
a metabelian small profinite group $G$ such that $G/G^{\prime}G^{p}$ is
infinite iff $p\in\pi$. Also $G$ is not strongly complete and $G^{\prime}$ is
not closed.
\end{proposition}

\begin{proof}
For distinct primes $p$ and $q$ we construct a finite group $G_{p,q}$ as
follows. Set $B=B_{q}=C_{q}^{(4q)}$ (the elementary abelian group of order
$q^{4q}$), and for $H\leq B$ let $A_{H}$ be the $\mathbb{F}_{p}B$-module
$(B-1)\mathbb{F}_{p}B/(H-1)\mathbb{F}_{p}B$. Note that $A_{H}(H-1)=0$ and
$A_{H}(B-1)=A_{H}$, since $p\neq q$ implies that $(B-1)\mathbb{F}_{p}B$ is an
idempotent ideal.

Put
\begin{align*}
M_{p,q}  &  =\bigoplus\limits_{H\in\mathcal{S}_{3q}(B)}A_{H}\\
G_{p,q}  &  =M_{p,q}\rtimes B_{q}.
\end{align*}
Note that%
\[
G_{p,q}^{\prime}=[M_{p,q},G_{p,q}]=M_{p,q}.
\]
Writing $D_{H}=\bigoplus\limits_{H\neq L\in\mathcal{S}_{3q}(B)}A_{L}$ we see
that Lemma \ref{wordexample} applies for the word $w_{p}=[x,y]z^{p}$, and
infer that this word does not have width $q$ in $G_{p,q}$.

Now partition $\pi^{\prime}$ (the set of primes complementary to $\pi$) into
$\left\vert \pi\right\vert $ infinite subsets $\sigma(p)$ $(p\in\pi)$. Set%
\[
G=\prod\limits_{p\in\pi,q\in\sigma(p)}G_{p,q}\text{.}%
\]
If $p\in\pi$ then $w_{p}$ does not have width $q$ in $G_{p,q}$, hence also not
in $G$, for every $q\in\sigma(p)$. So $w_{p}$ has infinite width in $G$. It
follows that $w_{p}(G)=G^{\prime}G^{p}$ is not closed, and therefore has
uncountable index in $G$.

If $r\in\pi^{\prime}$ then $G_{\{r\}}$ contains $\prod\limits_{p\in\pi,r\neq
q\in\sigma(p)}G_{p,q}$ so $G^{r}$ is open and $G/G^{^{\prime}}G^{r}\cong
B_{r}$ is finite.

If $q\nmid m\in\mathbb{N}$ then $G_{p,q}^{m}\geq\left\langle B_{q}^{G_{p,q}%
}\right\rangle =G_{p,q}$. It follows that
\[
\overline{G^{m}}\geq\prod\limits_{p\in\pi,q\in\sigma(p),q\nmid m}G_{p,q}%
\]
and hence that $\overline{G^{m}}$ is open. It follows trivially (see Theorem
\ref{small} below) that $G$ is small.

If $p\in\pi$ then the same argument shows that $\overline{G^{p}}=G$, while $G$
has infinitely many normal subgroups of index $p;$ none of these is open so
$G$ is not strongly complete. Finally, if $G^{\prime}$ were closed then
$G^{\prime}G^{p}=G^{\prime}G_{\{p\}}$ would be closed, being the product of
two compact subsets of $G$, whence $G=\overline{G^{p}}\leq G^{\prime}G^{p}$.
This is false for $p\in\pi$, so $G^{\prime}$ is not closed. (This may seem
counter-intuitive since at first glance one expects $G^{\prime}$ to be the
product of the $M_{p,q}$: the point is that an element of $M_{p,q}$ may be the
product of about $4q$ commutators, and an infinite product of such elements
may fail to be a product of finitely many commutators in $G$.)
\end{proof}

\bigskip

The next example is taken from \cite{N2}, Theorem 4. For any group $S$ we
denote by $\mathcal{V}_{S}$ the group variety generated by $S$ (the class of
all groups that satisfy all laws of $S$). If $S$ is finite then $\mathcal{V}%
_{S}$ is finitely based, by the Oates-Powell Theorem (see \cite{HN}, 52.12).
It follows that $\mathcal{V}_{S}$ can be defined by a single word, $w_{S}$.
Then for any group $G$, the corresponding verbal subgroup is $\mathcal{V}%
_{S}(G)=w_{S}(G)$.

A finite group is \emph{anabelian} if it has no abelian composition factors. A
profinite group $G$ is anabelian if $G/N$ is anabelian for every open normal
subgroup $N$ of $G$.

\begin{proposition}
\label{anabex}Let $S$ be a non-abelian finite simple group of exponent $m$.
There exists an anabelian small profinite group $G$ such that neither
$\mathcal{V}_{S}(G)$ nor $G^{m}$ is closed. $G$ is not strongly complete.
\end{proposition}

\begin{proof}
Say $w_{S}$ is a word on $k$ variables. Let $(T_{n})_{n\in\mathbb{N}}$ be a
sequence of finite non-abelian simple groups of strictly increasing exponents,
all exceeding $m$\ (for example, large alternating groups). Since the free
group $F_{kn}$ has only finitely many normal subgroups of index $\left\vert
T_{n}\right\vert $, there exists $r(n)$ such that $T_{n}^{(r(n))}$ cannot be
generated by $kn$ elements. Put $B_{n}=T_{n}^{(r(n))},$ for each
$H\in\mathcal{S}_{kn}(B_{n})$ let $\Omega_{H}$ be the $B_{n}$-set $H\setminus
B_{n}$, and let $M_{H}=S^{\Omega_{H}}$, a $B_{n}$-group where $B_{n}$ acts by
permuting the factors.

Let $M_{n}=\prod\limits_{H\in\mathcal{S}_{kn}(B_{n})}M_{H}$ (direct product)
and set
\[
G_{n}=M_{n}\rtimes B_{n}=S\wr_{\Omega}B_{n},
\]
the permutational wreath product where $\Omega$ is the disjoint union of the
transitive $G$-sets $\Omega_{H}$.

Let $H\in\mathcal{S}_{kn}(B_{n})$. Then $M_{H}=A_{H}\times C_{H}$ where
$A_{H}\cong S$ is the factor corresponding to $H$ in $\Omega_{H}$ and $C_{H}$
is the product of the remaining factors, and the conditions of Lemma
\ref{wordexample} are fulfilled on putting $D_{H}=C_{H}\times%
{\textstyle\prod\nolimits_{L\neq H}}
M_{L}$, both for $w=w_{S}$ and for $w=x^{m}$. Also%
\[
w_{S}(G_{n})\geq G_{n}^{m}\geq\left\langle (B_{n}^{m})^{G_{n}}\right\rangle
=\left\langle B_{n}{}^{G_{n}}\right\rangle =G_{n}%
\]
(the final equality holds because for each $H$ we have $\left\vert \Omega
_{H}\right\vert \geq2$ and $S$ is perfect).

We conclude that $w_{S}$ does not have width $n,$ and $x^{m}$ does not have
width $kn$, in $G_{n}$. Hence each of these words has infinite width in%
\[
G=\prod\limits_{n=1}^{\infty}G_{n}\text{,}%
\]
and so neither $\mathcal{V}_{S}(G)=w_{S}(G)$ nor $G^{m}$ is closed.

Let $q\in\mathbb{N}$. Then $T_{n}^{q}=T_{n}$ for all but finitely many $n$. As
above it follows that $G_{n}^{q}=G_{n}$ for all but finitely many $n$, and
hence (as above) that $\overline{G^{q}}$ is open in $G$, and finally that $G$
is therefore small.

That $G$ is not strongly complete follows from Theorem \ref{SW}, below.
\end{proof}

\bigskip

Different examples of small but not strongly complete groups were given in
\cite{N}, Proposition 27 and in an earlier version of this note.

\section{Small groups}

Write $s_{n}(G)$ to denote the number of (open) subgroups of index at most $n$
in a (pro)finite group $G.$ Thus a profinite group $P$ is small if and only if
$s_{n}(P)$ is finite for each $n$; this is equivalent to the statement: there
is a function $f:\mathbb{N}\rightarrow\mathbb{N}$ such that $s_{n}(G)\leq
f(n)$ for every $G\in\mathcal{F}(P)$ and all $n$.

\begin{theorem}
\label{small}A profinite group $P$ is small if and only if $\overline{P^{m}%
}\vartriangleleft_{o}P$ for every $m\in\mathbb{N}$.
\end{theorem}

Thus $P$ is small if and only if for each $m\in\mathbb{N}$ there exists $k(m)$
such that%
\begin{equation}
\forall Q\in\mathcal{F}(P):\left\vert Q/Q^{m}\right\vert \leq k(m).
\label{exptcond}%
\end{equation}
Equivalently: $\mathcal{F}(P)$ contains only finitely many groups of exponent
$m$.

This has a curious number-theoretic interpretation: with Chebotarev's Theorem
(\cite{K}, Theorem 1.116) it yields

\begin{corollary}
Let $S$ be a finite set of primes and let $m\in\mathbb{N}$. Then there are
only finitely many finite Galois extensions $K$ of $\mathbb{Q}$ such that
\emph{(1)} all primes ramified in $K$ are in $S$ and \emph{(2)} almost all
primes have residue degree at most $m$ in $K$.
\end{corollary}

In one direction, Theorem \ref{small} is obvious: every open subgroup of index
at most $n$ contains $\overline{P^{m}}$ where $m=n!$, so $s_{n}(P)\leq
s_{n}(P/\overline{P^{m}})<\infty$.

The other direction lies deeper; it generalizes the positive solution to the
Restricted Burnside Problem, which can be formulated as the statement:
$\overline{F^{m}}\vartriangleleft_{o}F$ for every $m\in\mathbb{N}$ when $F$ is
a finitely generated free profinite group.

It is proved in much the same way, bearing in mind the slightly different
hypothesis. Since $\overline{P^{m}}$ is the intersection of all
$N\vartriangleleft_{o}P$ with $P^{m}\leq N$, it will follow from the next
result, on taking $f(n)=s_{n}(P)$:

\begin{theorem}
\label{T3}Let $f:\mathbb{N}\rightarrow\mathbb{N}$ be a function and let
$m\in\mathbb{N}$. If $G$ is a finite group such that $G^{m}=1$ and
$s_{n}(G)\leq f(n)$ for all $n$ then $\left\vert G\right\vert \leq\nu(m,f)$, a
number depending only on $f$ and $m$.
\end{theorem}

For the rest of this section all groups will be finite. For a group $G$ let
$h^{\ast}(G)$ denote the minimal lengthof a chain of normal subgroups
$1=G_{0}\leq G_{1}<\ldots<G_{n}=G$ such that each factor $G_{i}/G_{i-1}$ is
either nilpotent or semisimple (here, a \emph{semisimple} \emph{group} means a
direct product of non-abelian simple groups). Classic results of Hall and
Higman, recalled in Section \ref{genfit} below, imply

\begin{theorem}
\label{hh} If $G^{m}=1$ then $h^{\ast}(G)\leq\eta(m)$, a number depending only
on $m$.
\end{theorem}

(Take $\eta(m)=2\delta(m)$ in Theorem \ref{hhbound}.)

Now let $G$ be a group satisfying the hypotheses of Theorem \ref{T3}.

\medskip

\noindent\emph{Case 1}. Suppose that $\left\vert G\right\vert =p^{e}$ for some
prime $p$, and that $\left\vert G/G^{\prime}G^{p}\right\vert =p^{d}$. Then
$p^{d-1}\leq s_{p}(G)\leq f(p)$ so $d\leq\lambda(p):=\left\lceil 1+\log
_{p}f(p)\right\rceil $. Now $G$ can be generated by $d$ elements, and then
Zelmanov's theorem \cite{Z1}, \cite{Z2} gives $\left\vert G\right\vert
\leq\beta(\lambda(p),m)$, a number depending only on $f(p)$ and $m$.

$\smallskip$

\noindent\emph{Case 2}. Suppose that $G$ is nilpotent. Say $m=p_{1}^{e_{1}%
}\ldots p_{r}^{e_{r}}$. Then from Case 1 we see that%
\[
\left\vert G\right\vert \leq\prod_{i=1}^{r}\beta(\lambda(p_{i}),m):=\nu
_{\mathrm{nil}}(m,f).
\]

\noindent\emph{Case 3}. Suppose that $G$ is semisimple. The result of
\cite{J}, with CFSG, shows that there are only finitely many non-abelian
simple groups $S$ such that $S^{m}=1$; call them $S_{1},\ldots,S_{k}$ and put
$t_{i}=\left\vert S_{i}\right\vert $. Now $G\cong\prod S_{i}^{(c_{i})}$ for
some $c_{i}\geq0$. Clearly $c_{i}\leq s_{t_{i}}(G)\leq f(t_{i})$ for each $i$,
and so%
\[
\left\vert G\right\vert \leq\prod_{i=1}^{k}t_{i}^{f(t_{i})}:=\nu_{\mathrm{ss}%
}(m,f).
\]

So far, we have shown that if $h^{\ast}(G)=1$ then%
\[
\left\vert G\right\vert \leq\max\{\nu_{\mathrm{nil}}(m,f),~\nu_{\mathrm{ss}%
}(m,f)\}:=\nu_{1}(m,f),
\]
say. Now let $q>1$ and suppose inductively that for each $h<q$, and every
function $g$, we have found a number $\nu_{h}(m,g)$ such that for any group
$H$ satisfying $h^{\ast}(H)\leq h$, $H^{m}=1$ and $s_{n}(H)\leq g(n)$ for all
$n$ we have $\left\vert H\right\vert \leq\nu_{h}(m,g)$.

Define%
\[
\nu_{q}(m,f)=\nu_{1}(m,f)\cdot\nu_{q-1}(m,g_{m,f})
\]
where $g_{m,f}(n)=f(n.\nu_{1}(m,f))$. Suppose that $G$ with $G^{m}=1$ satifies
$s_{n}(G)\leq f(n)$ for all $n$ and that $h^{\ast}(G)\leq q$. Thus $G$ has a
normal subgroup $H$ with $h^{\ast}(H)\leq q-1$ such that $G/H$ is either
nilpotent or semisimple. Then $\left\vert G/H\right\vert \leq\nu_{1}(m,f)$,
and so for each $n$ we have
\[
s_{n}(H)\leq s_{n.\nu_{1}(m,f)}(G)\leq g_{m,f}(n).
\]
Therefore $\left\vert H\right\vert \leq\nu_{q-1}(m,g_{m,f})$, whence
$\left\vert G\right\vert =\left\vert H\right\vert \left\vert G/H\right\vert
\leq\nu_{q}(m,f)$.

Finally, set%
\[
\nu(m,f)=\nu_{\eta(m)}(m,f).
\]
If $G$ satisfies the hypotheses of Theorem \ref{T3} then $h^{\ast}(G)\leq
\eta(m)$ by Theorem \ref{hh} and so $\left\vert G\right\vert \leq\nu(m,f)$ as required.

\section{Strongly complete groups}

The property of being small is inherently `algebraically defined', in terms of
the subgroup-growth functions $s_{n}(G)$, and more succinctly in the remark
following Theorem \ref{small}. The definition of `strongly complete', on the
other hand, refers directly to non-open subgroups, which by their nature are
undetectable in the continuous finite quotients of a profinite group. The
following characterization, due to Smith and Wilson, is therefore remarkable.

An \emph{f-variety} is the group variety generated by a finite group.

\begin{theorem}
\label{SW}\emph{(\cite{SW}, Theorem 2)} A profinite group $G$ is strongly
complete if and only if $\mathcal{V}(G)\vartriangleleft_{o}G$ for every
f-variety $\mathcal{V}$.
\end{theorem}

If $N\leq G$ and $\left\vert G/N\right\vert =m$ then $N\geq\mathcal{V}(G)$
where $\mathcal{V}=\mathcal{V}_{\mathrm{Sym}(m)}$, so $s_{m}(G)=s_{m}%
(G/\mathcal{V}(G))$; thus we have

\begin{corollary}
Every strongly complete profinite group is small.
\end{corollary}

See also \cite{P}, Theorem 2.4, where this was first proved using an
ultrafilter construction.

Now $\mathcal{V}(G)$ is open if and only it is both closed and has finite
index in $G$. If $\mathcal{V=V}_{S}$ for a finite group $S$, it is defined by
a word $w_{S}$; let us call such a word an \emph{f-word}. Then $\mathcal{V}%
(G)$ is closed in $G$ if and only if $w_{S}$ has finite width in $G$; in that
case,
\[
\left\vert G:\mathcal{V}(G)\right\vert =\left\vert G:\overline{\mathcal{V}%
(G)}\right\vert =\sup_{Q\in\mathcal{F}(G)}\left\vert Q:w_{S}(Q)\right\vert .
\]

Thus we have the `algebraic definition': $G$ is strongly complete if and only
if for each f-word $w$ there exists $k(w)\in\mathbb{N}$ such that
\begin{equation}
\forall Q\in\mathcal{F}(G):\left\vert Q/w(Q)\right\vert \leq k(w)\text{ and
}w\text{ has width }k(w)\text{ in }Q. \label{verbalchar}%
\end{equation}

Smith and Wilson (loc.cit) establish another characterization, which is not
algebraic in my sense but nicely clarifies the relation between `strongly
complete' and `small': $G$ \emph{is strongly complete if and only if }$G$
\emph{has finitely many subgroups of each finite index, and this holds if and
only if} $G$ \emph{has only countably many subgroups of finite index}.

Now (\ref{verbalchar}) looks like a strengthening of (\ref{exptcond}), except
that the power words $x^{m}$ are not (usually) f-words, because infinite
Burnside groups exist. Could we use power words instead of f-words here? The
question has some plausibility because every \emph{finitely generated}
profinite group is indeed power-open (the power subgroups $G^{m}$ are open,
\cite{NS2}). On the other hand, it is clear that every power-open profinite
group is strongly complete.

\medskip

\textbf{Question} \textbf{1. }\emph{Is every strongly complete profinite group
power-open?}

\medskip

If so, we can replace the f-words $w$ in (\ref{verbalchar}) by the power words
$x^{m}$, $m\in\mathbb{N}$.

The following reduction was pointed out to me by John Wilson:

\begin{proposition}
\label{jsw reduction}\emph{(J. S. Wilson)} Suppose that $G$ is strongly
complete. If $H^{q}\vartriangleleft_{o}H$ for every $H\vartriangleleft_{o}G$
and every prime-power $q\mid m$ then $G^{m}\vartriangleleft_{o}G$.
\end{proposition}

\begin{proof}
There are only finitely many finite simple groups of exponent dividing $m,$
say $S_{1},\ldots,S_{t}$ (\cite{J}+CFSG). Let $\mathcal{V}$ denote the variety
generated by $S_{1}\times\cdots\times S_{t}$.

Let $\eta(m)$ be the number given by Theorem \ref{hh}, so every finite group
of exponent dividing $m$ has a normal series of length $\eta(m)$ with each
factor either semisimple or nilpotent. Let $k=\eta(m)s$, where $m$ is
divisible by $s$ primes. Then every finite group of exponent dividing $m$ has
a normal series of length $k$ with each factor either in $\mathcal{V}$ or of
exponent $q$ for some prime-power $q\mid m$. It follows by a standard inverse
limit argument that every locally finite group of exponent dividing $m$ has
such a normal series. Now the main theorem of \cite{NS2} implies that
$G/G^{m}$ is locally finite; hence there is a normal series%
\[
G=G_{0}\geq G_{1}\geq\cdots\geq G_{k}=G^{m}%
\]
such that for each $i$, either $\mathcal{V}(G_{i})\leq G_{i+1}$ or $G_{i}%
^{q}\leq G_{i+1}$ for some prime-power $q\mid m$.

Let $i\in\{0,\ldots,k\}$ be maximal such that $G_{i}$ is open in $G$. Suppose
that $i<k$. Then $G_{i}$ is again strongly complete, so if $\mathcal{V}%
(G_{i+1})\leq G_{i}$ then $G_{i+1}\vartriangleleft_{o}G_{i}$ by Theorem
\ref{SW}, whence $G_{i+1}\vartriangleleft_{o}G$, contradiction. If
$G_{i}/G_{i+1}$ has exponent $q$ for some prime-power $q\mid m$ then
$G_{i+1}\vartriangleleft_{o}G_{i}$ by hypothesis, whence $G_{i+1}%
\vartriangleleft_{o}G$, again a contradiction. It follows that $i=k$ and so
$G^{m}\vartriangleleft_{o}G$.
\end{proof}

\bigskip

Thus it will suffice to answer Question 1 for \emph{prime-power} subgroups. In
some cases this is feasible:

\begin{theorem}
\label{anabpowers}\emph{(N. Nikolov) Let }$G$ be an anabelian profinite group.
Then $G=G^{q}$ for every prime-power $q$.
\end{theorem}

\begin{proof}
Suppose that $q$ is odd. Theorem 3 of \cite{N2} says that the word $x^{q}$ has
bounded width $l(q)$ in every finite anabelian group. In unpublished work
(personal communication), Nikolov proves the same statement for $q$ any power
of $2$. It follows in either case that $x^{q}$ has finite width $l(q)$ in $G$,
whence $G^{q}$ is closed. Now if $G^{q}\leq N\vartriangleleft_{o}G$ then $G/N$
is a finite anabelian group of prime-power order, whence $N=G$. the result follows.
\end{proof}

\bigskip

With Proposition \ref{jsw reduction} this gives

\begin{theorem}
Let $G$ be an anabelian profinite group. Then $G$ is strongly complete if and
only if $G$ is power-open.
\end{theorem}

At the other extreme we could consider prosoluble groups. Question 1 is still
open in this case, but the following may be relevant:

\begin{lemma}
Suppose that $G$ is strongly complete and prosoluble. Let $q=p^{n},$ $p$ a
prime, and let $P$ be a Sylow pro-$p$ subgroup of $G$. If $G^{q}$ is not
closed then $P_{1}:=P\cap\overline{G^{q}}\vartriangleleft_{o}P$ and $P_{1}$
has an infinite perfect quotient $P_{1}/(P\cap G^{q})$.
\end{lemma}

\begin{proof}
$\overline{G^{q}}\vartriangleleft_{o}G$ by Theorem \ref{small}, so
$P_{1}\vartriangleleft_{o}P$. Now $G$ has a Hall pro-$p^{\prime}$-subgroup
$H$, and $G^{q}P\geq HP=G$. So $\overline{G^{q}}=G^{q}P_{1}$ and so
\[
\frac{P_{1}}{P\cap G^{q}}=\frac{P_{1}}{P_{1}\cap G^{q}}\cong\frac
{\overline{G^{q}}}{G^{q}}.
\]
The latter is infinite and perfect, because $\overline{G^{q}}$ is strongly
complete and an abelian group of finite exponent is residually finite.
\end{proof}

\bigskip

Thus a negative answer to Question 1 would imply a positive answer to

\medskip

\textbf{Question 2.} \emph{Does there exist a pro-}$p$\emph{ group with a
nontrivial perfect quotient}?

\medskip

This is apparently unkown; the answer is probably `yes', but it seems quite hard.

To summarize some of the above:

\begin{theorem}
Let $G$ be a profinite group. The following conditions are equivalent to $G$
being strongly complete.

\emph{i.} if $G$ is a \emph{pro-}$p$ group: $G$ is finitely generated; or, $G$
is small; or, $\overline{G^{\prime}G^{p}}$ is open; or, $G^{\prime}G^{p}$ is open;

\emph{ii.} if $G$ is \emph{pronilpotent}: $G$ is small; or, each Sylow
subgroup of $G$ is finitely generated;

\emph{iii.} if $G$ is \emph{prosoluble}: $H^{\prime}H^{p}\vartriangleleft
_{o}H$ for every $H\vartriangleleft_{o}G$ and every prime $p$;

\emph{iv. }if $G$ is \emph{anabelian}: $G^{m}$ is open for every
$m\in\mathbb{N}$; or, $\left\vert G/G^{m}\right\vert $ is finite for every
$m\in\mathbb{N}$.
\end{theorem}

\begin{proof}
Most of this appears above, or follows easily. Let me sketch the argument for
(iii), where $G$ is prosoluble. Note that $H^{\prime}H^{p}=w(H)$ where
$w=[x,y]z^{p}$ define the variety generated by $C_{p}$, so if $G$ is strongly
complete and $H\vartriangleleft_{o}G$ then $H$ is strongly complete and $w(H)$
is open by Theorem \ref{SW}. For the converse, suppose that $G$ is not
strongly comlete and let $N$ be a normal subgroup of $G$ of minimal finite
index such that $N$ is not open. Then $G/N$ is a finite soluble group, by
Hall's characterization of finite soluble groups as those having a Hall
$p^{\prime}$-subgroup for every prime $p$: indeed, if $Q$ is a Hall
pro-$p^{\prime}$ subgroup of $G$ then $QN/N$ is a Hall $p^{\prime}$-subgroup
of $G/N$. Now let $H/N$ be a minimal normal subgroup of $G/N$. Then
$H\vartriangleleft_{o}G$ and $H^{\prime}H^{p}\leq N<H$ for some prime $p$; so
$H^{\prime}H^{p}$ is not open in $H$.
\end{proof}

\bigskip

\textbf{Remark} Theorem \ref{SW} does have a direct analogue for small groups:

\begin{theorem}
\label{easysmall}A profinite group $G$ is small if and only if $\overline
{\mathcal{V}(G)}\vartriangleleft_{o}G$ for every f-variety $\mathcal{V}$.
\end{theorem}

Of course this is an immediate corollary of Theorem \ref{small}, since if
$\mathcal{V=V}_{Q}$ where $Q$ has exponent $m$ then $G^{m}\leq\mathcal{V}(G)$.
However it is worth mentioning because it is completely elementary. To prove
it directly we argue exactly as in the proof of Theorem \ref{small}, quoting
Proposition \ref{fvht} (see Section \ref{genfit} below) in place of Theorem
\ref{hh}.

\section{The `congruence kernel'}

Let $G$ be a profinite group. Considered as an abstract group, $G$ has a
profinite completion $\widehat{G}$, and the identity map on $G$ induces a
natural continuous epimorphism $\pi:\widehat{G}\rightarrow G$.

The `congruence kernel' of $G$ is $C(G)=\ker\pi$. Note that $C(G)$ is the
projective limit%
\[
C(G)=\underset{\leftarrow}{\lim}\,\overline{N}/N
\]
where $N$ runs over normal subgroups of finite index in $G$. Thus $G$ is
strongly complete if and only if $C(G)=1$.

\begin{theorem}
\label{t1}If $C(G)$ is small then $G$ is strongly complete.
\end{theorem}

\begin{proof}
Assume that $C=C(G)$ is small. First we prove that $G$ is small.

Suppose for a contradiction that $G$ has infinitely many open normal subgroups
of index $n$. It is then easy to see that there exist an open normal subgroup
$H$ of $G$, a finite simple group $Q$ of order $\leq n$ and a continuous
epimorphism $\pi:H\rightarrow P=Q^{\mathbb{N}}$. For each non-principal
ultrafilter $\mathcal{U}$ on $\mathbb{N}$ let $\psi_{\mathcal{U}}:H\rightarrow
Q$ be the induced map onto the ultrapower $P/\mathcal{U}\cong Q$, and set
$K_{\mathcal{U}}=\ker\psi_{\mathcal{U}}$. Note that $K_{\mathcal{U}}$ contains
$\pi^{-1}(P_{0})$ where $P_{0}$ is the restricted direct power of $Q$ inside
$P$; if $S$ is any finite collection of non-principal ultrafilters, it follows
that%
\[
K_{S}:=\bigcap_{\mathcal{U}\in S}K_{\mathcal{U}}%
\]
is a dense normal subgroup of finite index in $H$. Thus $K_{S}$ contains a
normal subgroup $N$ of finite index in $G$ and $K_{S}\overline{N}=H$.
Therefore $H/K_{S}\cong\overline{N}/(\overline{N}\cap K_{S})$ is a continuous
image of $C$; say $H/K_{S}\cong C/M$ where $M$ is open and normal in $C$.

Let $\mathcal{V}$ be the variety generated by $Q$. Then $\overline
{\mathcal{V}(C)}\vartriangleleft_{o}C$ by Theorem \ref{easysmall}. Now
$\mathcal{V}(H)\leq K_{S}$ so $\overline{\mathcal{V}(C)}\leq M$ and so
$\left\vert H/K_{S}\right\vert \leq\left\vert C:\overline{\mathcal{V}%
(C)}\right\vert <\infty$. Choosing the set $S$ so as to maximize $\left\vert
H/K_{S}\right\vert $, we see that $K_{\mathcal{U}}\geq K_{S}$ for every
non-principal ultrafilter $\mathcal{U}$. Thus there are only finitely many
possibilities for $K_{\mathcal{U}}$.

Now it is easy to see that $K_{\mathcal{U}}$ determines $\mathcal{U}$; indeed,
for $V\subseteq\mathbb{N}$ we have%
\[
V\in\mathcal{U}\Longleftrightarrow K_{\mathcal{U}}\supseteq\pi^{-1}\left\{
f:\mathbb{N}\rightarrow Q\mid f(V)=\{1\}\right\}  .
\]
But the number of non-principal ultrafilters is infinite, so we have our contradiction.

Now fix an f-variety $\mathcal{V}$ and put $W=\mathcal{V}(G)$. If $W\leq
N\vartriangleleft_{f}G$ then $\overline{N}/N$ is a continuous image of $C$,
and as above we may infer that $\left\vert \overline{N}/N\right\vert
\leq\left\vert C:\overline{\mathcal{V}(C)}\right\vert <\infty$. We choose such
an $N$ so as to maximize $\left\vert \overline{N}/N\right\vert $.

Suppose that $W\leq M\vartriangleleft_{f}G$. Put $D=N\cap M$. Then
$\overline{D}N=\overline{N}$ so $\left\vert \overline{D}/(N\cap\overline
{D})\right\vert =\left\vert \overline{N}/N\right\vert $ and as $N\cap
\overline{D}\geq D$ it follows that $N\cap\overline{D}=D$. There are countably
many possibilities for $\overline{D}$, since $G$ is small; and given $D$,
there are finitely many possibilities for $M$. Thus there are countably many
possibilities for $M$.

Since there are countably many f-varieties it follows that $G$ has countably
many normal subgroups of finite index. The result follows by \cite{SW},
Theorem 2.
\end{proof}

\section{Generalized Fitting height: a reminder\label{genfit}}

In this section all groups are finite. The \emph{generalized Fitting subgroup}
of a group $G$ is $F^{\ast}(G)=FE$ where $F=F(G)$ is the Fitting subgroup and
$E=E(G)$ is the largest quasi-semisimple normal subgroup of $G$ (to say that
$E$ is \emph{quasi-semisimple} means that $E$ is perfect and $E/Z(E)$ is a
product of simple groups); $E$ is more usually defined as the subgroup
generated by the components of $G$, the quasisimple subnormal subgroups (that
this is equivalent is a small exercise). It is alway the case that $F\cap
E=\mathrm{Z}(E)$ and $E/\mathrm{Z}(E)$ is semisimple; see \cite{A}, Chapter
11. Thus $F^{\ast}/F$ is semisimple.

The \emph{generalized Fitting height }$h(G)$ of $G$ is defined by:
\[
h(1)=0,\qquad h(G)=1+h(G/F^{\ast}(G)).
\]
It is not hard to see that $h(G)$ is the minimal length of a series of normal
subgroups from $1$ to $G$ such that each factor is the product of a nilpotent
normal subgroup and a quasi-semisimple normal subgroup; it follows that $h$ is
sub-additive on group extensions.

The first result is elementary. For a group $Q$ the variety generated by $Q$
is denoted $\mathcal{V}_{Q}$.

\begin{proposition}
\label{fvht}For each finite group $Q$ there is an integer $m(Q)$ such that
$G\in\mathcal{V}_{Q}$ implies $h(G)\leq m(Q)$.
\end{proposition}

\begin{proof}
We define $m(Q)$ recursively: set $m(1)=0$ and suppose that $m(L)$ has been
found for every group $L$ with $\left\vert L\right\vert <\left\vert
Q\right\vert $.

If $G$ is a finite group in $\mathcal{V}_{Q}$ then $G$ is a section of
$Q^{(n)}$ for some finite $n$, so $h(G)\leq h(H)$ where $H\leq Q^{(n)}$. It
will suffice to find an upper bound for $h(H)$.

Let $M$ be a maximal normal subgroup of $Q$ and put $X=H\cap M^{(n)}$. Then
$X\in\mathcal{V}_{M}$ and $H/X\cong HM^{(n)}/M^{(n)}\in\mathcal{V}_{Q/M}$, so
if $M>1$ we have%
\[
h(H)\leq h(H/X)+h(X)\leq m(Q/M)+m(M).
\]
Thus if $Q$ is not simple we may define $m(Q)$ to be the infimum of
$m(Q/M)+m(M)$ where $M$ ranges over the maximal normal subgroups of $Q$.

Now suppose that $Q$ is simple. Write $\pi_{i}:H\rightarrow Q$ for the
projection to the $i$th factor in the product and set $L_{i}=\ker\pi_{i}$. Say
$H\pi_{i}=Q$ for $1\leq i\leq r$ and $H\pi_{i}=T_{i}<Q$ for $r<i\leq n$ (here
$r$ may be $0$ or $n$). Put $X=L_{1}\cap\ldots\cap L_{r}$. Then $H/X\cong
Q^{(t)}$ for some $t\leq r$ and $X\leq P:=T_{r+1}\times\cdots\times T_{n}$.

Now let $a=\max\{h(T)\mid T<Q\}$. Then $P$ has a series of normal subgroups
$1=A_{0}\leq B_{1}\leq A_{1}\leq\cdots\leq B_{a}\leq A_{a}=P$ with
$B_{i}/A_{i-1}$ nilpotent and $A_{i}/B_{i}$ semisimple. Say $S_{1}%
,\ldots,S_{s}$ are all the non-abelian composition factors of proper subgroups
of $Q$. Then%
\[
B_{i}=B_{i0}\leq B_{i1}\leq\cdots\leq B_{is}=A_{i}%
\]
where each $B_{ij}$ is normal in $P$ and $B_{ij}/B_{i(j-1)}\cong
S_{j}^{(n_{ij})}$. Intersecting with $X$ we obtain a normal series%
\[
\ldots A_{i-1}\cap X\leq B_{i}\cap X=X_{i0}\leq X_{i1}\leq\cdots\leq
X_{is}=A_{i}\cap X\ldots
\]
such that $(B_{i}\cap X)/(A_{i-1}\cap X)$ is nilpotent and
\[
\frac{X_{ij}}{X_{i(j-1)}}\cong\frac{B_{i(j-1)}(X\cap B_{ij})}{B_{i(j-1)}}%
\in\mathcal{V}(S_{j}),
\]
for each $i$ and $j$.

It follows that
\[
h((A_{i}\cap X)/(A_{i-1}\cap X))\leq1+m(S_{1})+\cdots+m(S_{s})=b
\]
say, and hence that $h(X)\leq ab$. As $H/X$ is semisimple we may therefore
define $m(Q)=1+ab$.
\end{proof}

\bigskip

The next result is not elementary: it depends on CFSG -- more precisely, it
needs the Schreier Conjecture and the Odd Order Theorem. It also depends on
the Hall-Higman Theorem (which it more or less implies, in a weak sense).

\begin{theorem}
\label{hhbound}For each $q\in\mathbb{N}$ there is an integer $\delta(q)$ such
that $G^{q}=1$ implies $h(G)\leq\delta(q)$.
\end{theorem}

\begin{proof}
Setting $\delta(1)=0$ we may suppose that $q>1$ and that $\delta(q^{\prime})$
has been defined for all $q^{\prime}<q$. Let $G$ be a group satisfying
$G^{q}=1$.

If $q$ is a prime power then $G$ is nilpotent and $h(G)\leq1$. Otherwise, let
$p$ be an odd prime divisor of $q=p^{e}r$ where $p\nmid r$.

Suppose first that $G$ is soluble. According to Theorem A of \cite{HH}, $G$
has $p$-length $l\leq2e+1$; so $G$ has a normal series%
\[
1=P_{0}\leq N_{0}<P_{1}<\cdots<P_{l}\leq N_{l}=G
\]
with each $P_{i}/N_{i-1}$ a $p$-group and $N_{i}^{r}\leq P_{i}$. It follows
that%
\[
h(G)\leq l(1+\delta(r)).
\]

Next, suppose that $\mathrm{Fit}(G)=1$ and let $M=F^{\ast}(G)$. Then
$M=S_{1}\times\cdots\times S_{n}$ is a product of non-abelian simple groups.
Let $L$ be the kernel of the induced permutation action of $G$ on the set
$\{S_{1},,\ldots,S_{n}\}$. Since $\mathrm{C}_{G}(M)=1$ (because $\mathrm{Fit}%
(G)=1$) we see that $L/M$ embeds into $\mathrm{Out}(S_{1})\times\cdots
\times\mathrm{Out}(S_{n})$, whence $L/M$ is soluble by the Schreier Conjecture.

The Odd Order Theorem ensures that $S_{1}$ has even order, and hence that $q$
is even. A simple argument, given below, shows that $G^{q/2}\leq L$. It
follows that%
\[
h(G)\leq1+l(1+\delta(r))+\delta(q/2).
\]

In general, let $H$ be the soluble radical of $G$. Then $\mathrm{Fit}(G/H)=1.$
Applying the two previous cases we deduce that $h(G)\leq\delta(q)$ where%
\[
\delta(q)=1+2l(1+\delta(r))+\delta(q/2).
\]

\emph{Proof that} $G^{q/2}\leq L$ (copied from the proof of \cite{HH}, Theorem
4.4.1). Suppose that the claim is false. Say $2^{e}=t$ exactly divides $q$.
Then there exists $g\in G$ with $g^{2^{e}}=1$ and $g^{2^{e-1}}\notin L$. Thus
$g$ has order $t$ modulo $L$. Hence $g$ in its conjugation action has a cycle
of length $t$ on $\{S_{1},,\ldots,S_{n}\}$, say $(S_{1},,\ldots,S_{t})$. Let
$x\in S_{1}$ be an element of order $2$. Then $S_{1}^{(xg)^{i}}=S_{1+i}$
centralizes $S_{1}$ for $1\leq i<t$, so for $h\in S_{1}$ we have%
\[
h^{(xg)^{t}}=h^{xg.g^{t-1}}=h^{x}.
\]
Choosing $h\in S_{1}\smallsetminus\mathrm{C}_{S_{1}}(x)$ we infer that
$(xg)^{t}\neq1$. But $(xg)^{t}=(x,x^{g},\ldots,x^{g^{t-1}})\in S_{1}%
\times\cdots\times S_{t}$ is an element of order $2$, so the order of $xg$ is
exactly $2t$; this contradicts $x^{q}=1$.
\end{proof}

\bigskip

It is not known whether the generalized Fitting height of all finite groups in
an arbitrary non-trivial variety is uniformly bounded; it would suffice to
settle this for soluble groups: see \cite{Kh}, Problem 2, Theorem 6, Theorem 7.

\end{document}